\documentclass[reqno,11pt]{amsart}
\usepackage{amsmath, amssymb, amsthm, mathrsfs}
\usepackage{geometry}
\usepackage{color}
\usepackage{setspace}
\usepackage[hidelinks]{hyperref}

\allowdisplaybreaks

\newtheorem{theorem}{Theorem}[section]
\newtheorem{lemma}[theorem]{Lemma}
\newtheorem{proposition}[theorem]{Proposition}
\newtheorem{definition}[theorem]{Definition}
\newtheorem{remark}[theorem]{Remark}

\numberwithin{equation}{section}

\begin{document}

\title[wave-MGT system with logarithmic nonlinearity]{Global existence and finite-time blow-up for a strongly damped wave--MGT system with fully subcritical logarithmic nonlinearity}

\author{Tae Gab Ha}

\address{Department of Mathematics and Institute of Pure and Applied Mathematics, Jeonbuk National University, Jeonju 54896, Republic of Korea}

\email{tgha@jbnu.ac.kr}



\begin{abstract}
We study a coupled strongly damped wave--Moore--Gibson--Thompson (MGT) system with logarithmic source $f(u)=|u|^{\gamma-2}u\ln|u|$
in the full Sobolev-subcritical range $2<\gamma<2^*:=2n/(n-2)$. The main difficulty is that, in this full range, the logarithmic nonlinearity does not admit the standard compactness and difference estimates available in the lower subcritical regime. Using the augmented variable $w=v+\tau v_t$, we derive the exact energy-dissipation identity for the strongly damped system and exploit the associated coupled potential-well structure.

For the local theory, a Faedo--Galerkin scheme combined with a closed nonlinear differential inequality and a spatial domain-splitting argument yields local existence. The strong damping provides the additional regularity needed to prove uniqueness throughout the full subcritical range and to establish a continuation principle. Below the well depth $d_\alpha$, we then obtain a sharp dynamical dichotomy: solutions with initial data in the stable set exist globally and decay exponentially, whereas solutions with initial data in the unstable set blow up in finite time.

The blow-up argument reveals a structural cancellation specific to the augmented formulation: once the logarithmic contribution is reconstructed through the exact energy identity, the unfavorable MGT residual in Levine's concavity functional is absorbed algebraically. This allows the concavity method to close without imposing any additional sign condition on the initial velocities.
\end{abstract}

\maketitle

\section{Introduction}

We consider the following coupled system:
\begin{equation}\label{eq:main}
\begin{cases}
u_{tt} - \Delta u - \Delta u_t + \alpha(v + \tau v_t) = f(u), & \text{in } \Omega \times (0, \infty), \\
\tau v_{ttt} + v_{tt} - \Delta v - b\Delta v_t + \alpha u = 0, & \text{in } \Omega \times (0, \infty), \\
u = v = 0, & \text{on } \partial\Omega \times (0, \infty), \\
u(x,0) = u_0(x), \quad u_t(x,0) = u_1(x), & \text{in } \Omega, \\
v(x,0) = v_0(x), \quad v_t(x,0) = v_1(x), \quad v_{tt}(x,0) = v_2(x), & \text{in } \Omega,
\end{cases}
\end{equation}
where $\Omega \subset \mathbb{R}^n$ ($n \ge 3$) is a bounded domain with smooth boundary $\partial\Omega$. Here, $\tau > 0$ denotes the thermal relaxation time, $b$ is the acoustic damping parameter, and $\alpha$ represents the coupling strength. For $u \neq 0$, the source term is given by
\begin{equation*}
f(u)=|u|^{\gamma-2}u\ln|u|,
\qquad
2<\gamma<2^*:=\frac{2n}{n-2}.
\end{equation*}
We set $f(0)=0$ and $F(0)=0$ by continuous extension, where
\begin{equation*}
F(s)=\int_0^s f(r)\,dr
\end{equation*}
denotes the primitive of $f$.

The Moore--Gibson--Thompson (MGT) equation, a third-order partial differential equation in time, has become a fundamental model in high-intensity nonlinear acoustics. It incorporates finite thermal relaxation effects and thereby removes the infinite propagation paradox inherent in classical Fourier-based heat conduction models \cite{kaltenbacher2011wellposedness}, \cite{lasiecka2016moore}, \cite{pellicer2019wellposedness}. In recent years, its well-posedness, chaotic behavior \cite{conejero2015chaotic}, optimal control \cite{bucci2019feedback}, and structural connections with linear viscoelasticity \cite{dell2017moore}, \cite{marchand2012abstract} have been extensively studied. Recent studies have also examined memory effects, regularity issues, and nonlinear extensions for the MGT equation; see, for instance, Dell'Oro, Lasiecka, and Pata \cite{delloro2016critical}, Bucci and Pandolfi \cite{bucci2020regularity}, and Nikoli\'c and Said-Houari \cite{nikolic2021jmgt}. When a MGT component is coupled with a structural wave equation, one obtains a genuinely hybrid hyperbolic system in which propagation, damping, and relaxation effects interact in a nontrivial way.

The third-order acoustic equation now known as the Moore--Gibson--Thompson equation traces back to Stokes's 1851 study of thermal effects on sound propagation \cite{stokes1851examination}; its modern name refers to the later works of Moore and Gibson \cite{moore1960weak} and Thompson \cite{thompson1972compressible}. In contemporary nonlinear acoustics, a Jordan--Moore--Gibson--Thompson equation is obtained from the balance laws for a compressible heat-conducting fluid by replacing Fourier's law with a Maxwell--Cattaneo law \cite{cattaneo1958forme}, which incorporates a finite thermal relaxation time $\tau>0$; see \cite{nikolic2021jmgt} for a derivation of the JMGT equation and \cite{jordan2014second} for the corresponding thermally relaxing-gas framework. Its linearization about a homogeneous quiescent state has the form
\begin{equation*}
\tau\psi_{ttt}+\psi_{tt}-c^2\Delta\psi-b\Delta\psi_t=0,
\qquad
b=\delta+\tau c^2,
\end{equation*}
where $c>0$ is the sound speed and $\delta>0$ is the sound-diffusivity parameter. The nonlinear JMGT framework arises in high-intensity ultrasound, whereas the displayed linear equation is the corresponding small-perturbation MGT model. In the normalization used in \eqref{eq:main}, the coefficient of $\psi_{tt}$ is normalized to one and the sound speed is scaled to $c=1$; hence $\delta>0$ is equivalent to $b>\tau$, and $b-\tau=\delta$ is the normalized sound diffusivity.

The first equation in \eqref{eq:main} is a semilinear strongly damped wave equation, with $-\Delta u_t$ representing Kelvin--Voigt-type strong damping; related logarithmic-source models
include \cite{lian2020global}. After introducing the augmented variable $w=v+\tau v_t$, the interaction becomes the bounded symmetric zero-order pair $\alpha w$ and $\alpha u$, of the type used in abstract weak-coupling and indirect-stabilization theory \cite{alabau2002indirect}. The restriction $|\alpha|<\lambda_1$ is an analytical small-coupling condition ensuring coercivity of the associated quadratic form. The full system \eqref{eq:main} is not claimed to arise from a single unified physical derivation; rather, it is a mathematically coupled model combining two established scalar components. The specific indirectly damped wave--MGT architecture from which the present model is obtained was formulated and analyzed in the author's earlier work \cite{ha2026wavemgt}. The present paper studies its strongly damped extension, and its contribution is analytical rather than modeling-based.

The logarithmic nonlinearity in \eqref{eq:main} has a well-established background in mathematical physics and nonlinear evolution equations. Such nonlinearities arise in models from quantum mechanics \cite{bialynicki1975logarithmic} and inflationary cosmology \cite{gorka2009logarithmic}, while their mathematical study in evolution equations goes back to the classical work of Cazenave and Haraux \cite{cazenave1980equations}. More recently, they have been incorporated into wave-type equations with damping, memory, and acoustic boundary effects; see, for example, Peng and Zhang \cite{peng2024memory}. 

At the analytical level, the logarithmic factor complicates both the compactness passage for approximate nonlinearities and the difference estimates required for uniqueness near the Sobolev threshold. Consequently, many well-posedness and stabilization results have been obtained only in the lower subcritical range
\begin{equation*}
2<\gamma<\frac{2(n-1)}{n-2}.
\end{equation*}
This restriction is not merely a technical artifact. It reflects the difficulty of controlling the logarithmic nonlinear terms by the available energy bounds in approximation and difference arguments. For example, recent works on strongly damped wave equations \cite{di2020initial}, on viscoelastic equations \cite{ha2021viscoelastic}, and on models involving nonlinear damping or $p$-Laplacian operators \cite{wu2023global} remain confined to the lower subcritical range $2<\gamma<\frac{2(n-1)}{n-2}$. 

For the global-versus-blow-up classification pursued here, the relevant variational framework is the classical potential-well method, originating in the work of Sattinger \cite{sattinger1968global} and in the Payne--Sattinger theory \cite{payne1975saddle}, while Levine's concavity method \cite{levine1974instability} provides the complementary mechanism for finite-time blow-up. The method uses the Nehari functional and the associated well depth to divide the phase space below the variational threshold into invariant stable and unstable regions. The dynamical roles of these regions were developed further for parabolic and hyperbolic equations by Ikehata and Suzuki \cite{ikehata1996stable} and, for nonlinear damped wave equations with source terms, by Todorova \cite{todorova1999stable}. Liu \cite{liu2003vacuum} introduced a family of potential wells and the vacuum-isolating phenomenon for semilinear wave equations, and Liu and Zhao \cite{liu2006potential} extended this construction to broader semilinear
hyperbolic and parabolic classes. In the dissipative wave setting, Vitillaro \cite{vitillaro1999global} established positive-energy global nonexistence under suitable relations between the damping and source exponents, whereas Gazzola and Squassina \cite{gazzola2006global} combined potential-well invariance with weak or strong damping and obtained both unstable-set blow-up and examples of finite-time blow-up at high initial energy.

Applications of this framework to logarithmic nonlinearities have developed in both parabolic and hyperbolic settings. In the parabolic case, Chen, Luo, and Liu \cite{chen2015global} gave an early representative application of a family of potential wells to a logarithmic heat equation, combining the variational structure with the logarithmic Sobolev inequality. For wave equations with logarithmic sources, Ma and Fang \cite{ma2018energy} developed a family-of-wells analysis for a strongly damped model and derived energy-decay and infinite-time blow-up results. Lian and Xu \cite{lian2020global} treated below-depth and critical-depth regimes in the presence of weak and strong damping and also obtained a high-energy infinite-time blow-up result in the weakly damped case. These works, together with related models involving combined logarithmic--power nonlinearities \cite{lian2020product} and the further references cited in the surrounding discussion, constitute the closest scalar potential-well background for the present system. 

Within this landscape, the present paper studies a strongly damped coupled wave--MGT system throughout the full Sobolev-subcritical exponent range $2<\gamma<2^*$. The novelty does not lie in the abstract potential-well mechanism itself; rather, the classical stable/unstable-set dichotomy is implemented on the coupled variational geometry determined by $Q_\alpha$, with hybrid dissipation supplied by the wave and MGT components. Accordingly, within the below-well-depth regime $\mathcal E(0)<d_\alpha$, the paper establishes exponential decay in the stable set and finite-time blow-up in the unstable set.

\noindent\textbf{Relation to the indirectly damped wave--MGT study.}
The present system is the strongly damped extension of the indirectly damped wave--MGT system analyzed in the author's earlier work \cite{ha2026wavemgt}. At the level of the differential equations, the essential modification is the addition of the strong damping term $-\Delta u_t$ to the wave equation. Apart from this term and the enlarged exponent range, the two problems have the same MGT equation, the same logarithmic source, the same zero-order coupling through $v+\tau v_t$, the same homogeneous Dirichlet boundary conditions, and the same structural conditions $b>\tau>0$ and $|\alpha|<\lambda_1$.

\smallskip

\noindent\textbf{Common augmented and variational framework.}
Both papers use the augmented variable
\begin{equation*}
w=v+\tau v_t
\end{equation*}
and the corresponding phase state $(u,u_t,w,w_t,v_t)$. In both systems, this change of variables converts the coupling into the symmetric pair $\alpha w$ and $\alpha u$, while $v_t$ remains the dissipative MGT coordinate. The two papers also have the same stored-energy functional; in particular, its static part is determined by the quadratic form
\begin{equation*}
Q_\alpha(u,w)
=
\|\nabla u\|_2^2+\|\nabla w\|_2^2+2\alpha(u,w)
\end{equation*}
and the logarithmic potential
\begin{equation*}
F(s)=\frac1\gamma |s|^\gamma\ln|s|-\frac1{\gamma^2}|s|^\gamma.
\end{equation*}
Consequently, the Nehari functional $\mathcal I_\alpha$, the potential functional $\mathcal J_\alpha$, the Nehari manifold, the well depth $d_\alpha$, and the stable potential-well geometry are common to the two analyses. The formal unstable branch used in the present paper is likewise the standard complementary branch determined by the same functionals. At the methodological level, both works use a Faedo--Galerkin approximation, a weak-level justification of the exact energy identity, strong temporal continuity of the augmented state, and a first-contact argument for the potential-well constraint. Accordingly, no novelty is claimed here for the variable $w=v+\tau v_t$, the stored-energy functional, the form $Q_\alpha$, or the associated static potential-well geometry; these ingredients were already used in \cite{ha2026wavemgt}.

\smallskip

\noindent\textbf{Dissipation and full-range local theory.}
Although the stored-energy functional $\mathcal E$ is common to both models, their dissipation laws are different. The indirectly damped system in \cite{ha2026wavemgt} satisfies, for $0\le s\le t$,
\begin{equation*}
\mathcal E(t)
+(b-\tau)\int_s^t\|\nabla v_t(\sigma)\|_2^2\,d\sigma
=\mathcal E(s),
\end{equation*}
whereas the present strongly damped system satisfies
\begin{equation*}
\mathcal E(t)
+\int_s^t
\Bigl(
\|\nabla u_t(\sigma)\|_2^2
+(b-\tau)\|\nabla v_t(\sigma)\|_2^2
\Bigr)\,d\sigma
=\mathcal E(s).
\end{equation*}
The additional direct dissipation gives $u_t\in L^2(0,T;H_0^1(\Omega))$ on every finite existence interval. This is the key regularity behind the local well-posedness and continuation theory for arbitrary energy data throughout $2<\gamma<2^*$. By contrast, \cite{ha2026wavemgt} treats $2<\gamma<2(n-1)/(n-2)$ and proves global weak well-posedness for data in the stable well (Theorem~4.6 therein). More precisely, the energy-space uniqueness estimate there requires a number $\mu>0$ such that
\begin{equation*}
\gamma-2+\mu<\frac{2}{n-2},
\end{equation*}
which forces the lower exponent range. For the present system, if $z=u^1-u^2$ is the difference of two wave components, the strong damping yields $z_t\in L^2(0,T;H_0^1(\Omega))$, allowing the nonlinear difference estimate to close for every $\gamma<2^*$. The resulting arbitrary-data local theory and continuation principle are stated in Proposition~\ref{prop:local}.

\smallskip

\noindent\textbf{Different stable-side dynamics.}
For $0<|\alpha|<\lambda_1$, Theorem~5.2 of
\cite{ha2026wavemgt} proves convergence to the zero equilibrium only under the additional assumption that the full trajectory is relatively compact in the natural phase space. Its conclusion is therefore a conditional qualitative LaSalle convergence theorem. Moreover, Proposition~2.2 of that paper identifies wave-branch modal roots satisfying
\begin{equation*}
\operatorname{Re}s_{k,\pm}
=
-\frac{\alpha^2}{2(b-\tau)\lambda_k^2}
+O(\lambda_k^{-5/2})
\longrightarrow0^{-},
\end{equation*}
which is presented there as a high-frequency spectral obstruction to a uniform exponential-decay statement. In contrast, every solution of the present strongly damped system starting below $d_\alpha$ in the stable set satisfies an exponential energy estimate, with no orbit-precompactness hypothesis. The proof uses an adapted Lyapunov functional and relies essentially on the direct coercive dissipation $\|\nabla u_t\|_2^2$; see Theorem~\ref{thm:main_stable}.

\smallskip

\noindent\textbf{Unstable-side dynamics.}
\cite{ha2026wavemgt} does not analyze the unstable branch and contains no finite-time blow-up result. The present paper proves positive invariance of the unstable set below the well depth and finite-time blow-up for every corresponding solution. In the concavity argument, the MGT component produces an unfavorable residual term. Reconstructing the logarithmic contribution through the exact energy identity yields an algebraic cancellation that absorbs this residual and closes Levine's argument without an additional sign condition on the initial velocities; see Theorem~\ref{thm:blowup}.

Thus, the independent contribution of the present paper is not the augmented variable or the static potential-well construction. It consists in determining how the additional strong damping changes the analytical and dynamical theory: full Sobolev-subcritical local well-posedness and continuation for arbitrary energy data, the strong-damping uniqueness estimate, exponential stabilization throughout the stable well without an orbit-precompactness assumption, and finite-time blow-up in the unstable set through the exact-energy cancellation described above. All arguments are carried out in full in the present paper; no theorem from \cite{ha2026wavemgt} is invoked as a black box.

A recent result reaching the full Sobolev-subcritical range $2<\gamma<2^*$ for a logarithmic wave equation is \cite{ha2026logarithm}, which treats a single strongly damped wave equation. Extending that type of analysis to the present coupled wave--MGT system is substantially more delicate, because the system is asymmetric and combines different dissipative mechanisms. In particular, one must simultaneously handle the logarithmic compactness problem, the nonlinear uniqueness issue, and the coupling-induced obstruction in the blow-up argument.

The main purpose of this paper is to determine the global dynamics of \eqref{eq:main} throughout the full Sobolev-subcritical range $2<\gamma<2^*$. Building on the augmented formulation $w=v+\tau v_t$ and the static energy geometry shared with \cite{ha2026wavemgt}, we analyze the consequences specific to the direct strong damping $-\Delta u_t$: the full-range local theory, stable-set exponential decay, and unstable-set finite-time blow-up. The main analytical obstacles can then be isolated and resolved systematically.

Compared with the lower-range theory, the present problem involves three main difficulties:
\begin{enumerate}
    \item \textbf{Loss of compactness.}
    In the full range $2<\gamma<2^*$, the logarithmic source no longer admits the standard polynomial controls used in the lower subcritical regime. To obtain $m$-independent bounds for the Galerkin approximations, one must derive a closed nonlinear differential inequality that exploits the strong damping.

    \item \textbf{Uniqueness obstruction.}
    The derivative $f'(u)$ exhibits logarithmic growth, so the nonlinear difference term cannot be estimated directly in the natural energy space.

    \item \textbf{Concavity residual in the blow-up argument.}
    In the unstable regime, the MGT coupling produces an unfavorable negative residual term in the second derivative of Levine's concavity functional \cite{levine1974instability}. In standard approaches, such a term typically forces additional sign conditions on the initial velocities.
\end{enumerate}

Our analysis resolves these difficulties as follows:
\begin{itemize}
    \item \textbf{Compactness via closed bounds and domain splitting.}
    We derive a closed nonlinear differential inequality for the quadratic Galerkin energy, which yields uniform local bounds. We then combine this with a spatial domain-splitting argument to isolate the logarithmic singularity, obtain uniform integrability of the source term, and pass to the limit by Vitali's theorem.

    \item \textbf{Uniqueness via strong damping.}
    The strong damping yields the additional regularity
    \begin{equation*}
    z_t\in L^2(0,T;H_0^1(\Omega))
    \end{equation*}
    for the difference variable, which is precisely what is needed to control the nonlinear term throughout the full subcritical range.

    \item \textbf{Blow-up via exact-energy cancellation.}
    After reconstructing the logarithmic term through the exact energy identity, we obtain an algebraic cancellation that absorbs the unfavorable MGT residual. This allows Levine's concavity method to close without any auxiliary sign condition on the initial velocities.
\end{itemize}

The main results of the paper may be summarized as follows. First, we establish local well-posedness and a continuation principle in the natural phase spaces associated with both the original and augmented formulations. Next, for initial data below the well depth $d_\alpha$ and inside the stable set, we prove invariance of the stable set, global existence, and uniform exponential decay of the exact energy by means of an adapted Lyapunov functional. On the other hand, for initial data below the same threshold but in the unstable set, we prove invariance of the unstable set and finite-time blow-up. In this way, the coupled strongly damped wave--MGT system exhibits a sharp dichotomy between exponential stabilization and blow-up in the full Sobolev-subcritical regime.

The remainder of the paper is organized as follows. In Section 2, we introduce the augmented formulation, the exact energy, and the coupled potential-well structure, and we state the main results. Section 3 is devoted to local well-posedness, the continuation principle, and strong continuity of weak solutions. In Section 4, we prove global existence and exponential stabilization in the stable set. Finally, Section 5 treats the unstable set and establishes finite-time blow-up.

\section{Preliminaries and Main Results}

In this section, we introduce the augmented formulation, the exact energy, and the static potential-well structure associated with \eqref{eq:main}. We also collect the auxiliary estimates needed in the later analysis and conclude by stating the main results of the paper.

We denote the standard $L^p(\Omega)$-norm by $\|\cdot\|_p$ and the $L^2(\Omega)$ inner product by $(\cdot,\cdot)$. We define the phase space for the initial data as
\begin{equation}
\mathcal{H}_0
=
H_0^1(\Omega)\times L^2(\Omega)\times H_0^1(\Omega)\times H_0^1(\Omega)\times L^2(\Omega).
\end{equation}
Correspondingly, the phase space for the state variable $Y(t)=(u,u_t,w,w_t,v_t)$ in the augmented formulation is defined as
\begin{equation}
\mathcal{H}
=
H_0^1(\Omega)\times L^2(\Omega)\times H_0^1(\Omega)\times L^2(\Omega)\times H_0^1(\Omega).
\end{equation}

\noindent\textbf{Assumption (A).}
We impose the following structural conditions.
\begin{itemize}
    \item \textbf{(A1) MGT Stability:} $b>\tau>0$.
    \item \textbf{(A2) Coupling Strength:} $|\alpha|<\lambda_1$, where $\lambda_1$ is the first Dirichlet eigenvalue of $-\Delta$ in $H_0^1(\Omega)$.
    \item \textbf{(A3) Fully Subcritical Exponent:} $2<\gamma<2^*:=\frac{2n}{n-2}$.
\end{itemize}

\subsection{The Augmented Formulation and Weak Solution}

We begin by rewriting \eqref{eq:main} in an augmented form that reveals the underlying coupled energy structure. Introducing the change of variable $w=v+\tau v_t$, the system is equivalently reformulated as
\begin{equation}\label{eq:augmented}
\begin{cases}
u_{tt}-\Delta u-\Delta u_t+\alpha w=f(u),\\
w_{tt}-\Delta w-(b-\tau)\Delta v_t+\alpha u=0.
\end{cases}
\end{equation}

For given initial data $I_0:=(u_0,u_1,v_0,v_1,v_2)\in\mathcal{H}_0$, the corresponding augmented initial states mapped into $\mathcal{H}$ are defined by
\begin{equation}\label{eq:initial_w}
w_0:=v_0+\tau v_1\in H_0^1(\Omega),
\qquad
w_1:=v_1+\tau v_2\in L^2(\Omega).
\end{equation}

\begin{definition}[Weak Solution]\label{def:weak}
Let $T>0$. A state
\begin{equation*}
Y(t)=(u,u_t,w,w_t,v_t)
\end{equation*}
is called a weak solution of \eqref{eq:main} on $[0,T]$ if
\begin{align*}
&u,w \in L^\infty(0,T;H_0^1(\Omega))
\cap C_w([0,T];H_0^1(\Omega))
\cap C([0,T];L^2(\Omega)),\\
&u_t,w_t \in L^\infty(0,T;L^2(\Omega))
\cap C_w([0,T];L^2(\Omega)),\\
&v_t \in L^\infty(0,T;H_0^1(\Omega))
\cap C_w([0,T];H_0^1(\Omega))
\cap C([0,T];L^2(\Omega)),
\end{align*}
with additional regularity
\begin{equation*}
u_t\in L^2(0,T;H_0^1(\Omega)),
\qquad
u_{tt},\,w_{tt}\in L^2(0,T;H^{-1}(\Omega)).
\end{equation*}
The variable $v$ is recovered by
\begin{equation*}
v:=w-\tau v_t.
\end{equation*}
Moreover, $v$ satisfies
\begin{equation*}
\partial_t v = v_t
\qquad\text{in }\mathcal D'((0,T)\times\Omega).
\end{equation*}
The state attains the initial conditions and satisfies \eqref{eq:augmented} in the distributional sense.
\end{definition}

\begin{remark}
Since $w_t, v_t \in L^\infty(0,T;L^2(\Omega))$, the quantity
\begin{equation*}
v_{tt}:=\tau^{-1}(w_t-v_t)
\end{equation*}
is well defined in $L^\infty(0,T;L^2(\Omega))$. Thus the stronger regularity of the recovered variable $v$ used later in the continuation argument is obtained a posteriori from the augmented state.
\end{remark}

The total exact energy is defined as
\begin{equation}\label{eq:energy}
\mathcal{E}(t)
:=
\frac12\|u_t\|_2^2
+\frac12\|w_t\|_2^2
+\frac{\tau(b-\tau)}{2}\|\nabla v_t\|_2^2
+\frac12 Q_\alpha(u,w)
-\int_\Omega F(u)\,dx,
\end{equation}
where
\begin{equation*}
Q_\alpha(u,w)
=
\|\nabla u\|_2^2+\|\nabla w\|_2^2+2\alpha(u,w),
\qquad
F(u)=\int_0^u f(r)\,dr.
\end{equation*}
Under Assumption~(A2), the quadratic form $Q_\alpha$ is coercive:
\begin{equation}\label{eq:coercivity}
Q_\alpha(u,w)
\ge
c_\alpha\bigl(\|\nabla u\|_2^2+\|\nabla w\|_2^2\bigr),
\end{equation}
where
\begin{equation*}
c_\alpha:=1-\frac{|\alpha|}{\lambda_1}>0.
\end{equation*}

\subsection{Coupled Potential Well Structure}

We next introduce the static functionals that encode the potential-well geometry of the coupled system. These quantities will play a central role in the stable/unstable set dichotomy developed in Sections~4 and~5.

We define
\begin{align}
\mathcal{I}_\alpha(u,w)
&=
Q_\alpha(u,w)-\int_\Omega |u|^\gamma\ln|u|\,dx,\\
\mathcal{J}_\alpha(u,w)
&=
\frac12 Q_\alpha(u,w)-\int_\Omega F(u)\,dx.
\end{align}
These functionals satisfy the algebraic identity
\begin{equation}\label{eq:algebraic}
\mathcal{J}_\alpha(u,w)
=
\frac{\gamma-2}{2\gamma}Q_\alpha(u,w)
+\frac{1}{\gamma}\mathcal{I}_\alpha(u,w)
+\frac{1}{\gamma^2}\|u\|_\gamma^\gamma.
\end{equation}

We define the Nehari manifold by
\begin{equation*}
\mathcal{N}_\alpha
:=
\{(u,w)\neq(0,0)\mid \mathcal{I}_\alpha(u,w)=0\},
\end{equation*}
and the corresponding potential well depth by
\begin{equation*}
d_\alpha
:=
\inf_{(u,w)\in\mathcal{N}_\alpha}\mathcal{J}_\alpha(u,w).
\end{equation*}

\begin{lemma}[Positivity of Depth]\label{lem:depth}
There exists $\kappa_0>0$ such that $Q_\alpha(u,w)\ge \kappa_0$ uniformly on $\mathcal{N}_\alpha$. Consequently, $d_\alpha>0$.
\end{lemma}

\begin{proof}
Fix any $\eta\in(0,2^*-\gamma)$. Since $\ln s\le C_\eta s^\eta$ for all $s\ge1$, we have
\begin{equation*}
|s|^\gamma\ln|s|
\le
C_\eta |s|^{\gamma+\eta}
\qquad\text{for all } |s|>1,
\end{equation*}
while $|s|^\gamma\ln|s|\le0$ for $|s|\le1$. Hence, for any $(u,w)\in\mathcal{N}_\alpha$,
\begin{equation*}
Q_\alpha(u,w)
=
\int_\Omega |u|^\gamma\ln|u|\,dx
\le
\int_{\{|u|>1\}} |u|^\gamma\ln|u|\,dx
\le
C_\eta\|u\|_{\gamma+\eta}^{\gamma+\eta}.
\end{equation*}
By the Sobolev embedding $H_0^1(\Omega)\hookrightarrow L^{\gamma+\eta}(\Omega)$, there exists a constant $S_{\gamma+\eta}>0$ such that
\begin{equation*}
\|u\|_{\gamma+\eta}\le S_{\gamma+\eta}\|\nabla u\|_2.
\end{equation*}
Combining this with the coercivity estimate $Q_\alpha(u,w)\ge c_\alpha\|\nabla u\|_2^2$, we obtain
\begin{equation*}
Q_\alpha(u,w)
\le
C_\eta S_{\gamma+\eta}^{\gamma+\eta}\|\nabla u\|_2^{\gamma+\eta}
\le
C_\eta S_{\gamma+\eta}^{\gamma+\eta}c_\alpha^{-(\gamma+\eta)/2}
Q_\alpha(u,w)^{(\gamma+\eta)/2}.
\end{equation*}
Set
\begin{equation*}
C_*
:=
C_\eta S_{\gamma+\eta}^{\gamma+\eta}c_\alpha^{-(\gamma+\eta)/2},
\qquad
p:=\frac{\gamma+\eta}{2}>1.
\end{equation*}
Then
\begin{equation*}
Q_\alpha(u,w)\le C_*Q_\alpha(u,w)^p.
\end{equation*}
Since $(u,w)\in\mathcal{N}_\alpha$ and $(u,w)\neq(0,0)$, \eqref{eq:coercivity} implies $Q_\alpha(u,w)>0$. Dividing by $Q_\alpha(u,w)$ gives
\begin{equation*}
1\le C_*Q_\alpha(u,w)^{p-1},
\end{equation*}
hence
\begin{equation*}
Q_\alpha(u,w)\ge \kappa_0:=C_*^{-1/(p-1)}>0
\qquad\text{for all } (u,w)\in\mathcal{N}_\alpha.
\end{equation*}

Therefore, using $\mathcal{I}_\alpha(u,w)=0$ in \eqref{eq:algebraic}, we get
\begin{equation*}
\mathcal{J}_\alpha(u,w)
=
\frac{\gamma-2}{2\gamma}Q_\alpha(u,w)
+\frac{1}{\gamma^2}\|u\|_\gamma^\gamma
\ge
\frac{\gamma-2}{2\gamma}Q_\alpha(u,w)
\ge
\frac{\gamma-2}{2\gamma}\kappa_0.
\end{equation*}
Taking the infimum over $(u,w)\in\mathcal{N}_\alpha$, we conclude that
\begin{equation*}
d_\alpha\ge \frac{\gamma-2}{2\gamma}\kappa_0>0.
\end{equation*}
\end{proof}

The stable and unstable subsets of the phase space are defined by
\begin{align}
\mathcal{W}_\alpha
&=
\bigl\{(u,w)\in H_0^1(\Omega)\times H_0^1(\Omega)
\mid
\mathcal{J}_\alpha(u,w)<d_\alpha,\ \mathcal{I}_\alpha(u,w)>0
\bigr\}\cup\{(0,0)\},\\
\mathcal{U}_\alpha
&=
\bigl\{(u,w)\in H_0^1(\Omega)\times H_0^1(\Omega)
\mid
\mathcal{J}_\alpha(u,w)<d_\alpha,\ \mathcal{I}_\alpha(u,w)<0
\bigr\}.
\end{align}

\begin{lemma}[Local Positivity Near the Origin]\label{lem:local_positivity}
There exists a constant $\rho_0>0$ such that for any $(u,w)$ satisfying $Q_\alpha(u,w)\le \rho_0$, we have $\mathcal{I}_\alpha(u,w)\ge \frac14 Q_\alpha(u,w)$.
\end{lemma}

\begin{proof}
Fix any $\eta\in(0,2^*-\gamma)$ and let $p:=\frac{\gamma+\eta}{2}>1$. As in the proof of Lemma~\ref{lem:depth}, we have
\begin{equation*}
\int_\Omega |u|^\gamma\ln|u|\,dx
\le
\int_{\{|u|>1\}} |u|^\gamma\ln|u|\,dx
\le
C_\eta\|u\|_{\gamma+\eta}^{\gamma+\eta}.
\end{equation*}
Using again the Sobolev embedding $H_0^1(\Omega)\hookrightarrow L^{\gamma+\eta}(\Omega)$ and the coercivity estimate $Q_\alpha(u,w)\ge c_\alpha\|\nabla u\|_2^2$, we obtain
\begin{equation*}
\int_\Omega |u|^\gamma\ln|u|\,dx
\le
C_\eta S_{\gamma+\eta}^{\gamma+\eta}\|\nabla u\|_2^{\gamma+\eta}
\le
C_*Q_\alpha(u,w)^p,
\end{equation*}
where
\begin{equation*}
C_*:=C_\eta S_{\gamma+\eta}^{\gamma+\eta}c_\alpha^{-(\gamma+\eta)/2}.
\end{equation*}

Choose $\rho_0>0$ small enough so that $C_*\rho_0^{p-1}\le \frac34$. If $Q_\alpha(u,w)\le \rho_0$, then
\begin{align*}
\mathcal{I}_\alpha(u,w)
&=
Q_\alpha(u,w)-\int_\Omega |u|^\gamma\ln|u|\,dx\\
&\ge
Q_\alpha(u,w)-C_*Q_\alpha(u,w)^p\\
&=
\bigl(1-C_*Q_\alpha(u,w)^{p-1}\bigr)Q_\alpha(u,w)\\
&\ge
\bigl(1-C_*\rho_0^{p-1}\bigr)Q_\alpha(u,w)\\
&\ge
\frac14 Q_\alpha(u,w).
\end{align*}
\end{proof}

\begin{lemma}[Uniform Gap in the Stable Set]\label{lem:gap}
For any $(u,w)\in\mathcal{W}_\alpha$ satisfying $\mathcal{J}_\alpha(u,w)\le E_0<d_\alpha$, there exists a constant
\begin{equation*}
\delta_0
=
1-\left(\frac{E_0}{d_\alpha}\right)^{(\gamma-2)/\gamma}\in(0,1)
\end{equation*}
such that
\begin{equation}
\mathcal{I}_\alpha(u,w)\ge \delta_0 Q_\alpha(u,w).
\end{equation}
\end{lemma}

\begin{proof}
If $u\equiv0$, the inequality is trivial. Assume therefore that $u\not\equiv0$. Since $\mathcal{I}_\alpha(u,w)>0$ and $\lim_{\lambda\to\infty}\mathcal{I}_\alpha(\lambda u,\lambda w)=-\infty$, continuity provides $\lambda^*>1$ such that
\begin{equation*}
\mathcal{I}_\alpha(\lambda^*u,\lambda^*w)=0.
\end{equation*}
By the definition of $d_\alpha$, we have
\begin{equation*}
\mathcal{J}_\alpha(\lambda^*u,\lambda^*w)\ge d_\alpha.
\end{equation*}
Applying \eqref{eq:algebraic} and using $\mathcal{I}_\alpha(\lambda^*u,\lambda^*w)=0$, we obtain
\begin{equation}
d_\alpha
\le
\mathcal{J}_\alpha(\lambda^*u,\lambda^*w)
=
(\lambda^*)^\gamma
\left[
\frac{\gamma-2}{2\gamma}(\lambda^*)^{2-\gamma}Q_\alpha(u,w)
+\frac{1}{\gamma^2}\|u\|_\gamma^\gamma
\right].
\end{equation}
Since $\gamma>2$ and $\lambda^*>1$, we have $(\lambda^*)^{2-\gamma}<1$, and thus
\begin{equation*}
\mathcal{J}_\alpha(\lambda^*u,\lambda^*w)
<
(\lambda^*)^\gamma \mathcal{J}_\alpha(u,w).
\end{equation*}
It follows that
\begin{equation*}
d_\alpha<(\lambda^*)^\gamma E_0,
\qquad\text{that is,}\qquad
(\lambda^*)^{-\gamma}<\frac{E_0}{d_\alpha}.
\end{equation*}

On the other hand, the relation $\mathcal{I}_\alpha(\lambda^*u,\lambda^*w)=0$ implies
\begin{equation*}
\int_\Omega |u|^\gamma\ln|u|\,dx
\le
(\lambda^*)^{2-\gamma}Q_\alpha(u,w).
\end{equation*}
Therefore,
\begin{equation*}
\mathcal{I}_\alpha(u,w)
=
Q_\alpha(u,w)-\int_\Omega |u|^\gamma\ln|u|\,dx
\ge
\bigl(1-(\lambda^*)^{2-\gamma}\bigr)Q_\alpha(u,w).
\end{equation*}
Since
\begin{equation*}
(\lambda^*)^{2-\gamma}
=
\bigl((\lambda^*)^{-\gamma}\bigr)^{(\gamma-2)/\gamma}
<
\left(\frac{E_0}{d_\alpha}\right)^{(\gamma-2)/\gamma},
\end{equation*}
we conclude that
\begin{equation*}
\mathcal{I}_\alpha(u,w)\ge \delta_0 Q_\alpha(u,w).
\end{equation*}
\end{proof}

\subsection{Main Theorems}

The preceding lemmas provide the basic static ingredients of the coupled potential-well method: positivity of the depth, positivity near the origin, and a uniform gap inside the stable set. We are now in a position to state the main results proved in the subsequent sections.

\begin{proposition}[Local Well-Posedness and Continuation]\label{prop:local}
Suppose Assumption~(A) holds. For any $I_0\in\mathcal{H}_0$, there exists a maximal existence time $T_{\max}\in(0,\infty]$ such that \eqref{eq:main} admits a unique local weak solution $Y\in C([0,T_{\max});\mathcal{H})$. Furthermore, if $T_{\max}<\infty$, then
\begin{equation*}
\limsup_{t\uparrow T_{\max}}\|Y(t)\|_{\mathcal{H}}=\infty.
\end{equation*}
\end{proposition}

\begin{theorem}[Global Existence and Stabilization]\label{thm:main_stable}
Suppose Assumption~(A) holds. Let $I_0\in\mathcal{H}_0$. If $\mathcal{E}(0)<d_\alpha$ and the augmented state $(u_0,w_0)\in\mathcal{W}_\alpha$, then the local weak solution is global, that is, $T_{\max}=\infty$. The state remains in $\mathcal{W}_\alpha$ for all $t\ge0$, and there exist constants $C_0,\omega>0$ such that
\begin{equation*}
\mathcal{E}(t)\le C_0e^{-\omega t}
\qquad\text{for all } t\ge0.
\end{equation*}
\end{theorem}

\begin{theorem}[Finite-Time Blow-up]\label{thm:blowup}
Suppose Assumption~(A) holds. Let $I_0\in\mathcal{H}_0$. If $\mathcal{E}(0)<d_\alpha$ and the augmented state $(u_0,w_0)\in\mathcal{U}_\alpha$, then the solution blows up in finite time, that is, $T_{\max}<\infty$ and
\begin{equation*}
\limsup_{t\uparrow T_{\max}}
\bigl(\|\nabla u(t)\|_2^2+\|\nabla w(t)\|_2^2\bigr)
=
\infty.
\end{equation*}
\end{theorem}

\section{Local Well-Posedness and Strong Continuity}

\subsection{Closed A Priori Estimates and Local Existence}

To establish local existence, we employ the Faedo--Galerkin method. Let $\{w_j\}_{j=1}^\infty$ be a complete basis of eigenfunctions of $-\Delta$ in $H_0^1(\Omega)$, and define
\begin{equation*}
V_m:=\operatorname{span}\{w_1,\dots,w_m\}.
\end{equation*}
We seek approximate solutions of the form
\begin{equation*}
u^m(t)=\sum_{j=1}^m g_{jm}(t)w_j,\qquad
w^m(t)=\sum_{j=1}^m h_{jm}(t)w_j,\qquad
v^m(t)=\sum_{j=1}^m l_{jm}(t)w_j,
\end{equation*}
where the coefficient functions satisfy the finite-dimensional Galerkin system obtained by projecting the augmented system onto $V_m$. The Galerkin initial data are chosen by spectral projection onto $V_m$, namely,
\begin{equation*}
u_0^m=P_m u_0,\qquad u_1^m=P_m u_1,\qquad
w_0^m=P_m w_0,\qquad w_1^m=P_m w_1,\qquad
v_1^m=P_m v_1,
\end{equation*}
where $P_m$ denotes the orthogonal projection onto $V_m$.

Our goal is to derive an $m$-independent closed a priori estimate for the projected system. To this end, rather than working directly with the full energy, whose potential part does not have a definite sign, we introduce the exact quadratic structural energy
\begin{equation*}
\widetilde{\mathcal{E}}_m(t):=\frac{1}{2}\|u_t^m\|_2^2+\frac{1}{2}\|w_t^m\|_2^2
+\frac{\tau(b-\tau)}{2}\|\nabla v_t^m\|_2^2+\frac{1}{2} Q_\alpha(u^m,w^m).
\end{equation*}
Since $P_m$ is bounded in both $L^2(\Omega)$ and $H_0^1(\Omega)$, there exists a constant $M_0>0$, independent of $m$, such that
\begin{equation*}
\widetilde{\mathcal E}_m(0)\le M_0.
\end{equation*}
Testing the approximate equations with $u_t^m$ and $w_t^m$, respectively, we obtain
\begin{equation}\label{eq:exact_energy}
\frac{d}{dt}\widetilde{\mathcal{E}}_m(t)+\|\nabla u_t^m\|_2^2+(b-\tau)\|\nabla v_t^m\|_2^2
=\int_\Omega f(u^m)u_t^m\,dx.
\end{equation}
Thus, the dissipative structure is preserved at the Galerkin level, and the main issue reduces to estimating the nonlinear source term on the right-hand side.

Since $\gamma<2^*$, the logarithmic source satisfies the subcritical polynomial bound
\begin{equation*}
|f(s)|\le C\bigl(1+|s|^{2^*-1}\bigr).
\end{equation*}
Applying the generalized H\"older inequality together with the Sobolev embedding $H_0^1(\Omega)\hookrightarrow L^{2^*}(\Omega)$, we estimate
\begin{equation}\label{eq:bound f_1}
\int_\Omega |f(u^m)||u_t^m|\,dx
\le C\int_\Omega \bigl(1+|u^m|^{2^*-1}\bigr)|u_t^m|\,dx
\le C_0\bigl(1+\|\nabla u^m\|_2^{\,2^*-1}\bigr)\|\nabla u_t^m\|_2.
\end{equation}
By Young's inequality,
\begin{equation}\label{eq:bound f_2}
C_0\bigl(1+\|\nabla u^m\|_2^{\,2^*-1}\bigr)\|\nabla u_t^m\|_2
\le \frac{1}{2}\|\nabla u_t^m\|_2^2
+\frac{C_0^2}{2}\bigl(1+\|\nabla u^m\|_2^{\,2^*-1}\bigr)^2.
\end{equation}

Next, by the coercivity of $Q_\alpha$, we have
\begin{equation*}
Q_\alpha(u^m,w^m)\ge c_\alpha \|\nabla u^m\|_2^2,
\end{equation*}
and hence
\begin{equation*}
\|\nabla u^m\|_2^2\le \frac{2}{c_\alpha}\widetilde{\mathcal{E}}_m(t).
\end{equation*}
Therefore,
\begin{equation}\label{eq:bound f_3}
\bigl(1+\|\nabla u^m\|_2^{\,2^*-1}\bigr)^2
\le 2+2\|\nabla u^m\|_2^{\,2(2^*-1)}
\le \widetilde{C}_1\bigl(1+(\widetilde{\mathcal{E}}_m(t))^{2^*-1}\bigr).
\end{equation}
Substituting \eqref{eq:bound f_1}--\eqref{eq:bound f_3} into \eqref{eq:exact_energy}, we arrive at the closed nonlinear differential inequality
\begin{equation*}
\frac{d}{dt}\widetilde{\mathcal{E}}_m(t)\le \widetilde{C}_2\bigl(1+(\widetilde{\mathcal{E}}_m(t))^{2^*-1}\bigr).
\end{equation*}
Because $2^*>2$, we have $2^*-1>1$, and the right-hand side depends only on $\widetilde{\mathcal{E}}_m(t)$ itself. Since the initial energy satisfies $\widetilde{\mathcal{E}}_m(0)\le M_0$ uniformly in $m$, standard ODE comparison yields a time $T_0>0$, depending only on $M_0$ and $\widetilde{C}_2$ but independent of $m$, such that
\begin{equation*}
\sup_{t\in[0,T_0]}\widetilde{\mathcal{E}}_m(t)\le 2M_0.
\end{equation*}
Consequently, the Galerkin solutions are uniformly bounded in
\begin{equation*}
u^m,w^m\quad \text{in }L^\infty(0,T_0;H_0^1(\Omega)),
\qquad
u_t^m,w_t^m\quad \text{in }L^\infty(0,T_0;L^2(\Omega)),
\end{equation*}
with the corresponding control of $\nabla v_t^m$ inherited from the definition of $\widetilde{\mathcal{E}}_m$. These uniform local-in-time bounds provide the starting point for the compactness argument developed in the next subsection.

\subsection{Compactness via Domain Splitting}

We now pass to the limit in the Galerkin scheme. The main difficulty is the non-scale-invariant logarithmic source
\begin{equation*}
f(u^m)=|u^m|^{\gamma-2}u^m \ln |u^m|,
\end{equation*}
for which a direct global polynomial bound is not available in the full range $2<\gamma<2^*$. To treat the logarithmic source, we decompose the domain according to the size of $u^m$:
\begin{equation*}
\Omega=\Omega_1^m\cup\Omega_2^m,
\qquad
\Omega_1^m:=\{x \in \Omega \mid |u^m|\le 1\},
\quad
\Omega_2^m:=\{x \in \Omega \mid |u^m|>1\}.
\end{equation*}

On $\Omega_1^m$, the function $r\mapsto |r|^{\gamma-1}|\ln |r||$ is bounded on $[0,1]$, and hence $f(u^m)$ is uniformly bounded. On $\Omega_2^m$, we use the elementary estimate
\begin{equation*}
\ln |u^m|\le C_\varepsilon |u^m|^\varepsilon
\qquad \text{for } |u^m|>1,
\end{equation*}
valid for any sufficiently small $\varepsilon>0$. Consequently,
\begin{equation*}
|f(u^m)|\le C_\varepsilon |u^m|^{\gamma-1+\varepsilon}
\qquad \text{on } \Omega_2^m.
\end{equation*}

We choose $\varepsilon$ so that
\begin{equation*}
0<\varepsilon\le \varepsilon_*:=\frac{\gamma-1}{\gamma}(2^*-\gamma).
\end{equation*}
Since $\gamma<2^*$, we have $\varepsilon_*>0$. Moreover, this choice ensures
\begin{equation*}
(\gamma-1+\varepsilon)\frac{\gamma}{\gamma-1}
=
\gamma+\varepsilon\frac{\gamma}{\gamma-1}
\le 2^*.
\end{equation*}
Let
\begin{equation*}
q:=\frac{\gamma}{\gamma-1}>1.
\end{equation*}
Then, combining the above splitting with the uniform bound $u^m$ in $L^\infty(0,T_0;H_0^1(\Omega))$ obtained in Subsection 3.1 and the Sobolev embedding $H_0^1(\Omega)\hookrightarrow L^{2^*}(\Omega)$, we conclude that $f(u^m)$ is uniformly bounded in $L^\infty(0,T_0;L^q(\Omega))$. Since $q>1$, the family $\{f(u^m)\}$ is uniformly integrable on $\Omega\times(0,T_0)$.

On the other hand, the Galerkin bounds and the equation yield the compactness required for the state variable. By the Aubin--Lions--Simon lemma, after extracting a subsequence, we have
\begin{equation*}
u^m \to u \qquad \text{strongly in } C([0,T_0];L^p(\Omega))
\quad \text{for every } 2\le p<2^*,
\end{equation*}
and in particular almost everywhere on $\Omega\times(0,T_0)$.

Since $f$ is continuous, the almost everywhere convergence $u^m\to u$ implies $f(u^m)\to f(u)$ almost everywhere. Together with the uniform integrability established above, Vitali's convergence theorem yields
\begin{equation*}
f(u^m)\to f(u)\qquad \text{strongly in } L^1(0,T_0;L^1(\Omega)).
\end{equation*}
Since the conjugate exponent of $q$ is $q'=\gamma<2^*$, the Sobolev embedding
\begin{equation*}
H_0^1(\Omega)\hookrightarrow L^\gamma(\Omega)
\end{equation*}
implies the continuous embedding
\begin{equation*}
L^q(\Omega)\hookrightarrow H^{-1}(\Omega).
\end{equation*}
Hence the uniform $L^\infty(0,T_0;L^q(\Omega))$ bound established above shows that $\{f(u^m)\}$ is bounded in $L^\infty(0,T_0;H^{-1}(\Omega))$. By weak-$*$ compactness, after passing to a further subsequence,
\begin{equation*}
f(u^m)\rightharpoonup^*\chi
\qquad\text{in }L^\infty(0,T_0;H^{-1}(\Omega))
\end{equation*}
for some $\chi\in L^\infty(0,T_0;H^{-1}(\Omega))$. The strong $L^1(0,T_0;L^1(\Omega))$ convergence obtained above identifies $\chi=f(u)$ in the sense of distributions. Consequently,
\begin{equation*}
f(u^m)\rightharpoonup^*f(u)
\qquad\text{in }L^\infty(0,T_0;H^{-1}(\Omega)).
\end{equation*}
In particular, for every
$\Phi\in L^1(0,T_0;H_0^1(\Omega))$,
\begin{equation*}
\int_0^{T_0}
\langle f(u^m(t)),\Phi(t)\rangle_{H^{-1},H_0^1}\,dt
\longrightarrow
\int_0^{T_0}
\langle f(u(t)),\Phi(t)\rangle_{H^{-1},H_0^1}\,dt.
\end{equation*}
This justifies the passage to the limit in the nonlinear term of the time-integrated weak formulation. Therefore, the Galerkin sequence converges, up to a subsequence, to a local weak solution on $[0,T_0]$.

\subsection{Uniqueness via Strong Damping}

To prove uniqueness, let $(u^1,w^1,v^1)$ and $(u^2,w^2,v^2)$ be two local weak solutions corresponding to the same initial data. Define the difference variables
\begin{equation*}
z:=u^1-u^2,\qquad y:=w^1-w^2,\qquad r:=v^1-v^2,
\end{equation*}
and introduce the associated difference energy
\begin{equation*}
\mathcal Z(t):=\|z_t\|_2^2+\|y_t\|_2^2+\tau(b-\tau)\|\nabla r_t\|_2^2+Q_\alpha(z,y).
\end{equation*}
Since the weak-solution regularity does not justify direct testing of the second difference equation by $y_t$, we regularize the difference system by Steklov averaging in time. Testing the averaged equations by $(z^h)_t$ and $(y^h)_t$, respectively, integrating over $(s,t)$, and passing to the limit as $h\to0$, we obtain
\begin{align*}
\mathcal Z(t)-\mathcal Z(s)
&+2\int_s^t
\Bigl(
\|\nabla z_\xi(\xi)\|_2^2
+(b-\tau)\|\nabla r_\xi(\xi)\|_2^2
\Bigr)\,d\xi  \\
&=
2\int_s^t
\int_\Omega
\bigl(f(u^1(\xi))-f(u^2(\xi))\bigr)z_\xi(\xi)\,dx\,d\xi. 
\end{align*}
Equivalently, the differential relation holds for a.e. $t\in(0,T)$:
\begin{equation}\label{eq:diff_energy}
\frac{1}{2}\frac{d}{dt}\mathcal Z(t)+\|\nabla z_t\|_2^2+(b-\tau)\|\nabla r_t\|_2^2
=
\int_\Omega \bigl(f(u^1)-f(u^2)\bigr)z_t\,dx.
\end{equation}

The main difficulty in estimating the right-hand side lies in the logarithmic growth of the derivative of $f$. More precisely, for $s\neq 0$,
\begin{equation*}
f'(s)=|s|^{\gamma-2}\bigl((\gamma-1)\ln|s|+1\bigr),
\end{equation*}
so $|f'(s)|$ cannot be controlled merely by $|s|^{\gamma-2}$ at infinity. However, for any $\varepsilon>0$, there exists $C_\varepsilon>0$ such that
\begin{equation*}
|f'(s)|\le C_\varepsilon\bigl(1+|s|^{\gamma-2+\varepsilon}\bigr)
\qquad\text{for all }s\in\mathbb R.
\end{equation*}
Hence, by the Mean Value Theorem,
\begin{equation*}
|f(u^1)-f(u^2)|
\le
C_\varepsilon\bigl(1+|u^1|^{\gamma-2+\varepsilon}+|u^2|^{\gamma-2+\varepsilon}\bigr)|z|.
\end{equation*}

We split the right-hand side of \eqref{eq:diff_energy} into a linear-growth part and a higher-order nonlinear part. The constant term is controlled by the Cauchy--Schwarz and Poincar\'e inequalities:
\begin{equation}\label{eq:diff_linear}
C_\varepsilon\int_\Omega |z||z_t|\,dx
\le
C\|z\|_2\|z_t\|_2
\le
C\|\nabla z\|_2\|z_t\|_2
\le
C\mathcal Z(t).
\end{equation}

For the nonlinear part, in order to cover the full subcritical range
$\gamma<2^*$, we choose
\begin{equation*}
\varepsilon=2^*-\gamma>0,
\end{equation*}
so that
\begin{equation*}
\gamma-2+\varepsilon=2^*-2=\frac{4}{n-2}.
\end{equation*}
Applying the generalized H\"older inequality with exponents $\frac{n}{2}$, $\frac{2n}{n-2}$, $\frac{2n}{n-2}$, together with the Sobolev embedding $H_0^1(\Omega)\hookrightarrow L^{2^*}(\Omega)$, we obtain, for $i=1,2$,
\begin{equation}\label{eq:diff_cross}
\int_\Omega |u^i|^{\frac{4}{n-2}}|z||z_t|\,dx
\le
\|u^i\|_{2^*}^{\frac{4}{n-2}}\|z\|_{2^*}\|z_t\|_{2^*}
\le
C_M\|\nabla z\|_2\|\nabla z_t\|_2.
\end{equation}
The estimate can be closed over the full subcritical range precisely because of the strong damping regularity $\nabla z_t\in L^2(\Omega)$. Applying Young's inequality, we find
\begin{equation}\label{eq:diff_young}
C_M\|\nabla z\|_2\|\nabla z_t\|_2
\le
C\mathcal Z(t)+\frac12\|\nabla z_t\|_2^2.
\end{equation}
Substituting \eqref{eq:diff_linear}--\eqref{eq:diff_young} into \eqref{eq:diff_energy}, we obtain
\begin{equation*}
\frac{d}{dt}\mathcal Z(t)\le C_0\mathcal Z(t).
\end{equation*}
Since $\mathcal Z(0)=0$, Gronwall's lemma implies that $\mathcal Z(t)=0$ for all $t\in[0,T]$. Therefore,
\begin{equation*}
z\equiv 0,\qquad y\equiv 0,\qquad r\equiv 0,
\end{equation*}
which proves uniqueness of the local weak solution.

\subsection{Continuation Principle}

We now prove the continuation criterion stated in Proposition~\ref{prop:local}. The key observation is that the length of the local existence interval produced by the Faedo--Galerkin construction depends only on the size of the initial data in the phase space $\mathcal H_0$.

Suppose, for contradiction, that the maximal existence time satisfies
\begin{equation*}
T_{\max}<\infty,
\end{equation*}
while
\begin{equation*}
\limsup_{t\uparrow T_{\max}}\|Y(t)\|_{\mathcal H}\le M<\infty.
\end{equation*}
Then the augmented state
\begin{equation*}
Y(t)=(u(t),u_t(t),w(t),w_t(t),v_t(t))
\end{equation*}
remains uniformly bounded in $\mathcal H$ on $[0,T_{\max})$. Hence, we may choose a sequence $t_n\uparrow T_{\max}$ such that
\begin{equation*}
Y(t_n)\in \mathcal H
\qquad\text{for every }n,
\end{equation*}
with
\begin{equation*}
\sup_n \|Y(t_n)\|_{\mathcal H}\le C_M.
\end{equation*}

To restart the local theory, we reconstruct the corresponding variables at time $t_n$ by
\begin{equation*}
v(t_n):=w(t_n)-\tau v_t(t_n),\qquad
v_{tt}(t_n):=\tau^{-1}\bigl(w_t(t_n)-v_t(t_n)\bigr).
\end{equation*}
Therefore, the state $\bigl(u(t_n),u_t(t_n),v(t_n),v_t(t_n),v_{tt}(t_n)\bigr)$ belongs to $\mathcal H_0$.

Moreover, its $\mathcal H_0$-norm is uniformly bounded independently of $n$. Indeed, by the Poincar\'e inequality,
\begin{equation*}
\|\nabla v(t_n)\|_2
\le
\|\nabla w(t_n)\|_2+\tau\|\nabla v_t(t_n)\|_2
\le C_M,
\end{equation*}
and
\begin{equation*}
\|v_{tt}(t_n)\|_2
\le
\tau^{-1}\Bigl(\|w_t(t_n)\|_2+\|v_t(t_n)\|_2\Bigr)
\le
\tau^{-1}\Bigl(\|w_t(t_n)\|_2+C_P\|\nabla v_t(t_n)\|_2\Bigr)
\le C_M.
\end{equation*}
Thus, the reconstructed data are uniformly bounded in $\mathcal H_0$.

Since the local well-posedness theory depends only on the size of the initial data in $\mathcal H_0$, it can be restarted at each time $t_n$. Consequently, there exists a time $\delta(M)>0$, independent of $n$, such that the solution extends to the interval
\begin{equation*}
[t_n,t_n+\delta(M)].
\end{equation*}
For all sufficiently large $n$, we have
\begin{equation*}
t_n+\delta(M)>T_{\max},
\end{equation*}
which contradicts the maximality of $T_{\max}$.

Therefore, the assumption
\begin{equation*}
T_{\max}<\infty
\quad\text{and}\quad
\limsup_{t\uparrow T_{\max}}\|Y(t)\|_{\mathcal H}<\infty
\end{equation*}
is impossible. Hence,
\begin{equation*}
T_{\max}<\infty
\quad\Longrightarrow\quad
\limsup_{t\uparrow T_{\max}}\|Y(t)\|_{\mathcal H}=\infty.
\end{equation*}
This proves the continuation principle.

\subsection{Exact Energy Identity and Strong Continuity} \label{subsec:strong_continuity}

We next justify the exact energy identity for weak solutions and upgrade the weak continuity of the state to strong continuity in the augmented phase space. To regularize the time variable, we employ Steklov averages
\begin{equation*}
u^h(t):=\frac1h\int_t^{t+h}u(s)\,ds,
\qquad
w^h(t):=\frac1h\int_t^{t+h}w(s)\,ds.
\end{equation*}
Testing the time-averaged equations with $(u^h)_t$ and $(w^h)_t$, respectively, and then passing to the limit as $h\to 0$, we recover the exact energy identity for every $0\le s\le t<T_{\max}$:
\begin{equation}\label{eq:exact_energy_identity_limit}
\mathcal{E}(t)+\int_s^t\Bigl(\|\nabla u_t(\sigma)\|_2^2+(b-\tau)\|\nabla v_t(\sigma)\|_2^2\Bigr)\,d\sigma
=\mathcal{E}(s).
\end{equation}
In particular, the map $t\mapsto \mathcal{E}(t)$ is continuous on $[0,T_{\max})$.

To prove strong continuity in $\mathcal H$, we introduce the equivalent Hilbert norm
\begin{equation*}
\|Y(t)\|_*^2
:=
\|u_t(t)\|_2^2+\|w_t(t)\|_2^2+\tau(b-\tau)\|\nabla v_t(t)\|_2^2+Q_\alpha(u(t),w(t)).
\end{equation*}
By the definition of the exact energy, this norm satisfies the algebraic relation
\begin{equation*}
\|Y(t)\|_*^2
=
2\mathcal{E}(t)+2\int_\Omega F(u(t))\,dx.
\end{equation*}

Let $t_n\to t$ be any sequence. By the continuity of the exact energy established above,
\begin{equation*}
\mathcal{E}(t_n)\to \mathcal{E}(t).
\end{equation*}
On the other hand, the weak continuity
\begin{equation*}
Y\in C_w([0,T_{\max});\mathcal H)
\end{equation*}
implies
\begin{equation*}
u(t_n)\rightharpoonup u(t)\quad\text{weakly in }H_0^1(\Omega),
\qquad
u(t_n)\to u(t)\quad\text{strongly in }L^2(\Omega).
\end{equation*}
Since $\gamma<2^*$, we may choose
\begin{equation*}
0 < \eta < \min\{\gamma-2,\,2^*-\gamma\}.
\end{equation*}
For the logarithmic primitive
\begin{equation*}
F(s)=\frac1\gamma |s|^\gamma \ln|s|-\frac1{\gamma^2}|s|^\gamma,
\end{equation*}
we have the two-sided growth bound
\begin{equation*}
|F(s)|\le C_\eta\bigl(|s|^{\gamma-\eta}+|s|^{\gamma+\eta}\bigr)
\qquad\text{for all }s\in\mathbb R.
\end{equation*}
Because $u$ is bounded in $H_0^1(\Omega)$ and continuous in $L^2(\Omega)$, Sobolev
interpolation yields
\begin{equation*}
u\in C([0,T_{\max});L^{\gamma-\eta}(\Omega))
\cap
C([0,T_{\max});L^{\gamma+\eta}(\Omega)).
\end{equation*}
In particular, $u(t_n)\to u(t)$ in measure on $\Omega$.

Moreover, the sequence $\{u(t_n)\}$ is uniformly bounded in $H_0^1(\Omega)$, hence also in $L^{2^*}(\Omega)$ by the Sobolev embedding $H_0^1(\Omega)\hookrightarrow L^{2^*}(\Omega)$. Since $\gamma+\eta<2^*$, one may choose $p>1$ such that
\begin{equation*}
p(\gamma+\eta)\le 2^*.
\end{equation*}
The above growth estimate for $F$ then shows that $\{F(u(t_n))\}$ is uniformly bounded in $L^p(\Omega)$. Because $p>1$, this implies that $\{F(u(t_n))\}$ is uniformly integrable over $\Omega$. Combined with the almost everywhere convergence of a subsequence, Vitali's convergence theorem yields
\begin{equation*}
F(u(t_n))\to F(u(t))
\qquad\text{strongly in }L^1(\Omega).
\end{equation*}
Therefore,
\begin{equation*}
\int_\Omega F(u(t_n))\,dx \to \int_\Omega F(u(t))\,dx,
\end{equation*}
so the map
\begin{equation*}
t\mapsto \int_\Omega F(u(t))\,dx
\end{equation*}
is continuous on $[0,T_{\max})$. Consequently, the identity
\begin{equation*}
\|Y(t)\|_*^2
=
2\mathcal{E}(t)+2\int_\Omega F(u(t))\,dx
\end{equation*}
shows that $t\mapsto \|Y(t)\|_*^2$ is continuous. Since
\begin{equation*}
Y\in C_w([0,T_{\max});\mathcal H),
\end{equation*}
weak continuity together with norm continuity in the Hilbert space $\mathcal H$ implies
\begin{equation*}
Y\in C([0,T_{\max});\mathcal H).
\end{equation*}
In particular, the quadratic form
\begin{equation*}
t\mapsto Q_\alpha(u(t),w(t))
\end{equation*}
is continuous. Finally, using the algebraic identity
\begin{equation*}
|u|^\gamma\ln|u|
=
\gamma F(u)+\frac1\gamma |u|^\gamma,
\end{equation*}
we conclude that
\begin{equation*}
t\mapsto \mathcal{I}_\alpha(u(t),w(t))
\end{equation*}
is also continuous on $[0,T_{\max})$.

\section{Proof of Stabilization in the Stable Set}

\subsection{Invariance of the Stable Set}

We first show that the stable set is positively invariant along the flow. This fact will allow us to propagate the potential-well structure obtained in Section~2 throughout the entire lifespan of the solution.

\begin{proposition}[Invariance of the Stable Set]\label{prop:stable_invariance}
If $\mathcal{E}(0) < d_\alpha$ and $(u_0,w_0)\in \mathcal{W}_\alpha$, then the state $(u(t),w(t))$ remains in $\mathcal{W}_\alpha$ for all $t\in[0,T_{\max})$.
\end{proposition}

\begin{proof}
Suppose, for contradiction, that there exists a first time $t_*\in(0,T_{\max})$ at which the trajectory exits the stable set $\mathcal{W}_\alpha$. Since $Y\in C([0,T_{\max});\mathcal{H})$ by Subsection~\ref{subsec:strong_continuity}, the mappings
\begin{equation*}
t\mapsto \mathcal{I}_\alpha(u(t),w(t)),
\qquad
t\mapsto \mathcal{J}_\alpha(u(t),w(t))
\end{equation*}
are continuous on $[0,T_{\max})$. Therefore, at the boundary time $t_*$, one must have either
\begin{equation*}
\mathcal{J}_\alpha(u(t_*),w(t_*))=d_\alpha,
\end{equation*}
or
\begin{equation*}
\mathcal{I}_\alpha(u(t_*),w(t_*))=0
\quad\text{with}\quad
(u(t_*),w(t_*))\neq(0,0).
\end{equation*}

The first alternative is impossible. Indeed, by the exact energy identity, $\mathcal{E}(t)$ is non-increasing, and by definition $\mathcal{J}_\alpha(u,w)\le \mathcal{E}(t)$. Hence,
\begin{equation*}
\mathcal{J}_\alpha(u(t_*),w(t_*))
\le \mathcal{E}(t_*)
\le \mathcal{E}(0)
< d_\alpha.
\end{equation*}
This contradicts $\mathcal{J}_\alpha(u(t_*),w(t_*))=d_\alpha$.

Consider next the second alternative. If $\mathcal{I}_\alpha(u(t_*),w(t_*))=0$ and $(u(t_*),w(t_*))\neq(0,0)$, then $(u(t_*),w(t_*))\in \mathcal{N}_\alpha$. By the definition of the well depth, this implies
\begin{equation*}
\mathcal{J}_\alpha(u(t_*),w(t_*))\ge d_\alpha.
\end{equation*}
On the other hand, the same energy bound as above gives
\begin{equation*}
\mathcal{J}_\alpha(u(t_*),w(t_*))
\le \mathcal{E}(t_*)
\le \mathcal{E}(0)
<d_\alpha,
\end{equation*}
which is again a contradiction.

It remains to exclude the possibility that the trajectory leaves $\mathcal{W}_\alpha$ through the origin. If $(u(t_*),w(t_*))=(0,0)$, then by definition $(u(t_*),w(t_*))\in \mathcal{W}_\alpha$. Moreover, by continuity of $Q_\alpha(u,w)$, any sufficiently small perturbation away from the origin satisfies
\begin{equation*}
Q_\alpha(u(t),w(t))\le \rho_0
\end{equation*}
for $t$ near $t_*$. Lemma~\ref{lem:local_positivity} then yields
\begin{equation*}
\mathcal{I}_\alpha(u(t),w(t))
\ge \frac14 Q_\alpha(u(t),w(t))
\ge 0,
\end{equation*}
so the trajectory cannot continuously cross from $\mathcal{W}_\alpha$ into its complement through $(0,0)$.

Therefore, no exit time can occur, and the trajectory remains in $\mathcal{W}_\alpha$ for all $t\in[0,T_{\max})$.
\end{proof}

\subsection{Global Existence and Uniform Exponential Decay}

We now prove Theorem~\ref{thm:main_stable}. Once the invariance of the stable set is known, Lemma~\ref{lem:gap} provides a uniform positive gap for $\mathcal{I}_\alpha$, which in turn yields a coercive lower bound for the exact energy. This already implies global existence via the continuation principle. The exponential decay is then obtained from a Lyapunov functional adapted to the coupled $(u,w,v)$-structure.

\begin{proof}[Proof of Theorem \ref{thm:main_stable}]
By Proposition~\ref{prop:stable_invariance}, the trajectory satisfies
\begin{equation*}
(u(t),w(t))\in \mathcal{W}_\alpha
\qquad\text{for all }t\in[0,T_{\max}).
\end{equation*}
Setting $E_0:=\mathcal{E}(0)$, Lemma~\ref{lem:gap} gives the uniform estimate
\begin{equation*}
\mathcal{I}_\alpha(u(t),w(t))\ge \delta_0 Q_\alpha(u(t),w(t)),
\qquad t\in[0,T_{\max}),
\end{equation*}
where
\begin{equation*}
\delta_0=1-\Bigl(\frac{E_0}{d_\alpha}\Bigr)^{(\gamma-2)/\gamma}\in(0,1).
\end{equation*}

Substituting this bound into the algebraic decomposition of $\mathcal{J}_\alpha$, we obtain
\begin{equation}\label{eq:stable_energy_lower}
\begin{aligned}
\mathcal{E}(t)
&\ge \mathcal{J}_\alpha(u,w) \\
&= \frac{\gamma-2}{2\gamma}Q_\alpha(u,w)
 + \frac{1}{\gamma}\mathcal{I}_\alpha(u,w)
 + \frac{1}{\gamma^2}\|u\|_\gamma^\gamma \\
&\ge
\left(
\frac{\gamma-2}{2\gamma}+\frac{\delta_0}{\gamma}
\right)
Q_\alpha(u,w).
\end{aligned}
\end{equation}
Using the coercivity estimate \eqref{eq:coercivity}, together with the positivity of the kinetic part of $\mathcal{E}(t)$, we deduce
\begin{equation}\label{eq:stable_coercive_bound}
\mathcal{E}(t)\ge C_{\min}\|Y(t)\|_{\mathcal H}^2
\qquad\text{for all } t\in[0,T_{\max}).
\end{equation}
Since the exact energy is non-increasing, $\mathcal{E}(t)\le \mathcal{E}(0)$, and therefore $\|Y(t)\|_{\mathcal H}$ remains uniformly bounded on $[0,T_{\max})$. The continuation principle in Proposition~\ref{prop:local} then implies that
\begin{equation*}
T_{\max}=\infty.
\end{equation*}
Thus the solution is global.

It remains to prove the uniform exponential decay of the exact energy. For this purpose, we introduce the Lyapunov functional
\begin{equation*}
\mathcal{L}(t):=N\mathcal{E}(t)+\Phi_{\mathrm{pot}}(t)+3\tau \Phi_{\mathrm{kin}}(t),
\end{equation*}
where
\begin{equation}\label{eq:lyapunov_components}
\Phi_{\mathrm{pot}}(t)
:=\int_\Omega (u_tu+w_tw)\,dx,
\qquad
\Phi_{\mathrm{kin}}(t)
:=\int_\Omega
\left(
-w_tv_t-\frac12|\nabla v|^2+\frac12|v_t|^2
\right)\,dx.
\end{equation}

By \eqref{eq:stable_coercive_bound}, we have $\|Y(t)\|_{\mathcal H}^2\le C\mathcal{E}(t)$. Hence, by the Cauchy--Schwarz and Poincar\'e inequalities,
\begin{equation*}
|\Phi_{\mathrm{pot}}(t)|+|3\tau\Phi_{\mathrm{kin}}(t)|
\le C_\Phi \mathcal{E}(t).
\end{equation*}
Choosing $N>C_\Phi$, we infer that $\mathcal{L}(t)$ is equivalent to $\mathcal{E}(t)$ in the sense that there exist constants $c_1,c_2>0$ such that 
\begin{equation}\label{eq:L_equiv_E} 
c_1\mathcal{E}(t)\le \mathcal{L}(t)\le c_2\mathcal{E}(t)
\qquad\text{for all } t\ge0.
\end{equation}

All differentiations below are first justified at the Galerkin level and then passed to the limit. We begin with $\Phi_{\mathrm{kin}}$. A direct computation gives
\begin{equation*}\label{eq:phikin_derivative_1}
\begin{aligned}
\Phi'_{\mathrm{kin}}(t)
&=
-\bigl(\Delta w+(b-\tau)\Delta v_t-\alpha u,\,v_t\bigr)
-(v_t+\tau v_{tt},v_{tt})
-(\nabla v,\nabla v_t)
+(v_t,v_{tt}) \\
&=
(\nabla w,\nabla v_t)
+(b-\tau)\|\nabla v_t\|_2^2
+\alpha(u,v_t)
-\tau\|v_{tt}\|_2^2
-(\nabla v,\nabla v_t).
\end{aligned}
\end{equation*}
Since $w=v+\tau v_t$, we have $\nabla w-\nabla v=\tau \nabla v_t$, and therefore
\begin{equation}\label{eq:phikin_derivative_2}
\Phi'_{\mathrm{kin}}(t)
=
b\|\nabla v_t\|_2^2
+\alpha(u,v_t)
-\tau\|v_{tt}\|_2^2.
\end{equation}

Next, differentiating $\Phi_{\mathrm{pot}}$ and using the augmented system
\eqref{eq:augmented}, we obtain
\begin{equation}\label{eq:phipot_derivative}
\begin{aligned}
\Phi'_{\mathrm{pot}}(t)
&=
\|u_t\|_2^2+\|w_t\|_2^2+(u,u_{tt})+(w,w_{tt}) \\
&=
\|u_t\|_2^2+\|w_t\|_2^2
-\|\nabla u\|_2^2-(\nabla u,\nabla u_t)-\alpha(w,u)
+\int_\Omega |u|^\gamma \ln|u|\,dx \\
&\quad
-\|\nabla w\|_2^2-(b-\tau)(\nabla w,\nabla v_t)-\alpha(u,w) \\
&=
\|u_t\|_2^2+\|w_t\|_2^2
-\mathcal{I}_\alpha(u,w)
-(\nabla u,\nabla u_t)
-(b-\tau)(\nabla w,\nabla v_t).
\end{aligned}
\end{equation}

Combining \eqref{eq:phikin_derivative_2} and \eqref{eq:phipot_derivative}, and using the invariant gap
\begin{equation*}
\mathcal{I}_\alpha(u,w)\ge \delta_0Q_\alpha(u,w)
\ge \delta_0 c_\alpha
\bigl(\|\nabla u\|_2^2+\|\nabla w\|_2^2\bigr),
\end{equation*}
we arrive at
\begin{equation}\label{eq:combined_auxiliary_derivative}
\begin{aligned}
\Phi'_{\mathrm{pot}}(t)+3\tau\Phi'_{\mathrm{kin}}(t)
&\le
\|u_t\|_2^2+\|w_t\|_2^2
-\delta_0 c_\alpha
\bigl(\|\nabla u\|_2^2+\|\nabla w\|_2^2\bigr) \\
&\quad
-(\nabla u,\nabla u_t)
-(b-\tau)(\nabla w,\nabla v_t)
+3\tau b\|\nabla v_t\|_2^2 \\
&\quad
+3\tau\alpha(u,v_t)
-3\tau^2\|v_{tt}\|_2^2.
\end{aligned}
\end{equation}

We now estimate the right-hand side term of \eqref{eq:combined_auxiliary_derivative}. First, using $w_t=v_t+\tau v_{tt}$ together with Poincar\'e's inequality, we obtain
\begin{equation}\label{eq:kinetic_velocity_bound}
\|u_t\|_2^2+\|w_t\|_2^2
\le
C_P^2\|\nabla u_t\|_2^2
+2C_P^2\|\nabla v_t\|_2^2
+2\tau^2\|v_{tt}\|_2^2.
\end{equation}
Next, Young's inequality with
\begin{equation*}
\varepsilon_1=\varepsilon_2=\frac{\delta_0 c_\alpha}{4},
\qquad
\varepsilon_3=\frac{\delta_0 c_\alpha}{8},
\end{equation*}
gives
\begin{align}
|(\nabla u,\nabla u_t)|
&\le
\varepsilon_1\|\nabla u\|_2^2
+\frac{1}{4\varepsilon_1}\|\nabla u_t\|_2^2,
\label{eq:young_uut}\\
|(b-\tau)(\nabla w,\nabla v_t)|
&\le
\varepsilon_2\|\nabla w\|_2^2
+\frac{(b-\tau)^2}{4\varepsilon_2}\|\nabla v_t\|_2^2,
\label{eq:young_wvt}\\
|3\tau\alpha(u,v_t)|
&\le
\varepsilon_3\|\nabla u\|_2^2
+\frac{9\tau^2\alpha^2C_P^4}{4\varepsilon_3}\|\nabla v_t\|_2^2.
\label{eq:young_coupling}
\end{align}

Substituting \eqref{eq:kinetic_velocity_bound}--\eqref{eq:young_coupling} into \eqref{eq:combined_auxiliary_derivative}, we see that the positive term $2\tau^2\|v_{tt}\|_2^2$ is absorbed by the negative contribution $-3\tau^2\|v_{tt}\|_2^2$, leaving $-\tau^2\|v_{tt}\|_2^2$. Moreover, the gradient terms retain a strictly negative coercive part. More precisely,
\begin{equation}\label{eq:auxiliary_derivative_estimate}
\Phi'_{\mathrm{pot}}(t)+3\tau\Phi'_{\mathrm{kin}}(t)
\le
C_1\|\nabla u_t\|_2^2
+
C_2\|\nabla v_t\|_2^2
-\tau^2\|v_{tt}\|_2^2
-\frac{\delta_0 c_\alpha}{2}
\bigl(\|\nabla u\|_2^2+\|\nabla w\|_2^2\bigr),
\end{equation}
where
\begin{equation*}
C_1:=C_P^2+\frac{1}{\delta_0 c_\alpha},
\qquad
C_2:=3\tau b+2C_P^2+\frac{(b-\tau)^2}{\delta_0 c_\alpha}
+\frac{18\tau^2\alpha^2C_P^4}{\delta_0 c_\alpha}.
\end{equation*}

On the other hand, the exact energy identity yields
\begin{equation}\label{eq:energy_dissipation_stable}
\mathcal{E}'(t)
=
-\|\nabla u_t\|_2^2-(b-\tau)\|\nabla v_t\|_2^2.
\end{equation}
Therefore, differentiating $\mathcal{L}(t)$ and combining \eqref{eq:auxiliary_derivative_estimate} with \eqref{eq:energy_dissipation_stable}, we obtain
\begin{equation}\label{eq:L_derivative_preliminary}
\begin{aligned}
\mathcal{L}'(t)
&\le
-\bigl(N-C_1\bigr)\|\nabla u_t\|_2^2
-\bigl(N(b-\tau)-C_2\bigr)\|\nabla v_t\|_2^2 \\
&\quad
-\tau^2\|v_{tt}\|_2^2
-\frac{\delta_0 c_\alpha}{2}
\bigl(\|\nabla u\|_2^2+\|\nabla w\|_2^2\bigr).
\end{aligned}
\end{equation}

To close the estimate, we still need to control the negative potential part in $\mathcal{E}(t)$. For the logarithmic primitive,
\begin{equation*}
-F(s)=\frac{1}{\gamma^2}|s|^\gamma-\frac{1}{\gamma}|s|^\gamma\ln|s|,
\qquad s\in\mathbb R.
\end{equation*}
Near the origin, the dominant term behaves like $|s|^\gamma |\ln |s||$, which is bounded by $C|s|^2$ because $\gamma>2$. For large $|s|$, one has $-F(s)\le 0$. Hence there exists $C_F>0$ such that
\begin{equation}\label{eq:F_quadratic_bound}
-F(s)\le C_F|s|^2
\qquad\text{for all } s\in\mathbb R.
\end{equation}
Applying Poincar\'e's inequality and the coercivity of $Q_\alpha$, we deduce
\begin{equation}\label{eq:potential_control_by_Q}
-\int_\Omega F(u)\,dx
\le
C_F\|u\|_2^2
\le
C_FC_P^2\|\nabla u\|_2^2
\le
\frac{C_FC_P^2}{c_\alpha}Q_\alpha(u,w).
\end{equation}
Consequently, the exact energy is controlled by the quadratic terms:
\begin{equation}\label{eq:E_controlled_by_quadratic_terms}
\mathcal{E}(t)
\le
C_E\Bigl(
\|\nabla u_t\|_2^2
+\|\nabla v_t\|_2^2
+\tau^2\|v_{tt}\|_2^2
+\|\nabla u\|_2^2
+\|\nabla w\|_2^2
\Bigr).
\end{equation}

We now choose
\begin{equation*}
N>\max\left\{C_1,\frac{C_2}{b-\tau}\right\}.
\end{equation*}
Then every coefficient on the right-hand side of \eqref{eq:L_derivative_preliminary} is strictly negative, and \eqref{eq:E_controlled_by_quadratic_terms} implies the existence of a constant
$c_3>0$ such that
\begin{equation}\label{eq:L_prime_controls_E}
\mathcal{L}'(t)\le -c_3\mathcal{E}(t).
\end{equation}
Finally, using the upper bound in \eqref{eq:L_equiv_E}, namely $\mathcal{L}(t)\le c_2\mathcal{E}(t)$, we infer
\begin{equation*}
-\mathcal{E}(t)\le -\frac{1}{c_2}\mathcal{L}(t).
\end{equation*}
Substituting this into \eqref{eq:L_prime_controls_E}, we arrive at
\begin{equation}\label{eq:closed_Lyapunov_decay}
\mathcal{L}'(t)\le -\frac{c_3}{c_2}\mathcal{L}(t)=:-\omega \mathcal{L}(t),
\qquad \omega>0.
\end{equation}
Integration of \eqref{eq:closed_Lyapunov_decay} yields
\begin{equation*}
\mathcal{L}(t)\le \mathcal{L}(0)e^{-\omega t},
\qquad t\ge0.
\end{equation*}
Using again the equivalence \eqref{eq:L_equiv_E}, we conclude that
\begin{equation*}
\mathcal{E}(t)\le C_0e^{-\omega t}
\qquad\text{for all } t\ge0,
\end{equation*}
for some constant $C_0>0$. This proves the uniform exponential stabilization.
\end{proof}

\section{Dynamics in the Unstable Set: Finite-Time Blow-up}

\subsection{Invariance of the Unstable Set}

We first show that the unstable set is positively invariant along the flow. This will allow us to propagate the sign condition $\mathcal{I}_\alpha(u(t),w(t))<0$ throughout the lifespan of the solution and, at the same time, recover a uniform positive lower bound for the quadratic part $Q_\alpha$.

\begin{proposition}[Invariance of the Unstable Set]\label{prop:unstable_invariance}
Let the initial data satisfy $\mathcal{E}(0)<d_\alpha$ and $\mathcal{I}_\alpha(u_0,w_0)<0$. Let $Y(t)$ be the local weak solution on its maximal existence interval $[0,T_{\max})$. Then
\begin{equation*}
\mathcal{I}_\alpha(u(t),w(t))<0
\qquad\text{and}\qquad
Q_\alpha(u(t),w(t))\ge \kappa_0>0
\end{equation*}
for all $t\in[0,T_{\max})$.
\end{proposition}

\begin{proof}
Let
\begin{equation*}
\Sigma_+:=\{t\in[0,T_{\max}) : \mathcal{I}_\alpha(u(t),w(t))\ge 0\}.
\end{equation*}
Suppose, for contradiction, that $\Sigma_+$ is nonempty, and let $t_*:=\inf \Sigma_+>0$. By the strong continuity established in Subsection~\ref{subsec:strong_continuity}, the map $t\mapsto \mathcal{I}_\alpha(u(t),w(t))$ is continuous. Hence
\begin{equation*}
\mathcal{I}_\alpha(u(t_*),w(t_*))=0.
\end{equation*}

We distinguish two cases.

If $(u(t_*),w(t_*))\neq (0,0)$, then $(u(t_*),w(t_*))\in \mathcal{N}_\alpha$, and therefore, by the definition of the well depth,
\begin{equation*}
\mathcal{J}_\alpha(u(t_*),w(t_*))\ge d_\alpha.
\end{equation*}
On the other hand, the exact energy is non-increasing, and $\mathcal{J}_\alpha(u,w)\le \mathcal{E}(t)$ by definition. Thus,
\begin{equation*}
\mathcal{J}_\alpha(u(t_*),w(t_*))
\le \mathcal{E}(t_*)
\le \mathcal{E}(0)
< d_\alpha,
\end{equation*}
which is impossible.

If instead $(u(t_*),w(t_*))=(0,0)$, then $Q_\alpha(u(t_*),w(t_*))=0$. Since $Y\in C([0,T_{\max});\mathcal H)$, the map $t\mapsto Q_\alpha(u(t),w(t))$ is continuous. Hence there exists $\delta>0$ such that
\begin{equation*}
Q_\alpha(u(t),w(t))\le \rho_0
\qquad\text{for all } t\in(t_*-\delta,t_*].
\end{equation*}
Lemma~\ref{lem:local_positivity} then gives
\begin{equation*}
\mathcal{I}_\alpha(u(t),w(t))
\ge \frac14 Q_\alpha(u(t),w(t))
\ge 0
\qquad\text{for } t\in(t_*-\delta,t_*],
\end{equation*}
which contradicts the definition of $t_*$, since $\mathcal{I}_\alpha(u(t),w(t))$ must remain negative for $t<t_*$ close enough to $t_*$.

Therefore, the trajectory never exits the unstable set, and
\begin{equation*}
\mathcal{I}_\alpha(u(t),w(t))<0
\qquad\text{for all } t\in[0,T_{\max}).
\end{equation*}

It remains to prove the uniform lower bound for $Q_\alpha$. Fix $t\in[0,T_{\max})$, and define
\begin{equation*}
g(\lambda):=\mathcal{I}_\alpha(\lambda u(t),\lambda w(t))
=\lambda^2Q_\alpha(u(t),w(t))
-\lambda^\gamma\int_\Omega |u(t)|^\gamma \ln(\lambda|u(t)|)\,dx.
\end{equation*}
Since $\mathcal{I}_\alpha(u(t),w(t))<0$, we have $g(1)<0$. On the other hand, for $\lambda\in(0,1)$,
\begin{equation*}
g(\lambda)
=
\lambda^2\left(
Q_\alpha(u(t),w(t))
-\lambda^{\gamma-2}\int_\Omega |u(t)|^\gamma\ln|u(t)|\,dx
-\lambda^{\gamma-2}\ln\lambda\,\|u(t)\|_\gamma^\gamma
\right).
\end{equation*}
Because $\gamma>2$, we have
\begin{equation*}
\lambda^{\gamma-2}\to 0
\qquad\text{and}\qquad
\lambda^{\gamma-2}|\ln\lambda|\to 0
\qquad\text{as } \lambda\downarrow 0.
\end{equation*}
Moreover, $(u(t),w(t))\neq (0,0)$ implies $Q_\alpha(u(t),w(t))>0$. Hence the positive quadratic term dominates for sufficiently small $\lambda>0$, and so $g(\lambda)>0$ for all such $\lambda$.

By continuity, there exists $\lambda_t\in(0,1)$ such that $g(\lambda_t)=0$, namely,
\begin{equation*}
\mathcal{I}_\alpha(\lambda_tu(t),\lambda_tw(t))=0.
\end{equation*}
Thus $(\lambda_tu(t),\lambda_tw(t))\in\mathcal{N}_\alpha$, and Lemma~\ref{lem:depth} yields
\begin{equation*}
Q_\alpha(\lambda_tu(t),\lambda_tw(t))\ge \kappa_0.
\end{equation*}
Using the homogeneity of $Q_\alpha$, we conclude
\begin{equation*}
Q_\alpha(u(t),w(t))
=
\lambda_t^{-2}Q_\alpha(\lambda_tu(t),\lambda_tw(t))
\ge
\lambda_t^{-2}\kappa_0
>
\kappa_0.
\end{equation*}
This completes the proof.
\end{proof}

\subsection{Finite-Time Blow-up}

We now prove Theorem~\ref{thm:blowup}. The argument is based on Levine's concavity method. Compared with the classical wave equation, the main new difficulty is that the MGT component produces the unfavorable residual term $-2\tau(b-\tau)\|\nabla v_t\|_2^2$ in the second derivative of the concavity functional. The key observation is that, after reconstructing the logarithmic term through the exact energy identity, this residual is absorbed algebraically. This is the structural mechanism that closes the concavity argument without imposing any additional sign condition such as $\Psi'(0)>0$.

\begin{proof}[Proof of Theorem \ref{thm:blowup}]
We argue by contradiction. Assume that the solution is global, that is, $T_{\max}=\infty$. Let $T>0$ be an arbitrary terminal time. We define the
concavity functional
\begin{equation}\label{eq:psi}
\begin{aligned}
\Psi(t)
&=
\|u(t)\|_2^2+\|w(t)\|_2^2
+\int_0^t\Bigl(\|\nabla u(s)\|_2^2+(b-\tau)\|\nabla v(s)\|_2^2\Bigr)\,ds \\
&\quad
+(T-t)\Bigl(\|\nabla u_0\|_2^2+(b-\tau)\|\nabla v_0\|_2^2\Bigr)
+\beta(t+t_0)^2,
\end{aligned}
\end{equation}
where $\beta>0$ and $t_0>0$ are constants to be chosen later.

Differentiating with respect to time, and using
\begin{equation*}
\frac{d}{dt}\left(
\int_0^t \|\nabla u(s)\|_2^2\,ds
+
(T-t)\|\nabla u_0\|_2^2
\right)
=
\|\nabla u(t)\|_2^2-\|\nabla u_0\|_2^2
=
2\int_0^t(\nabla u,\nabla u_t)\,ds,
\end{equation*}
together with the analogous identity for the $v$-part, we obtain
\begin{equation}\label{eq:psi_single}
\Psi'(t)
=
2(u,u_t)+2(w,w_t)
+2\int_0^t\Bigl((\nabla u,\nabla u_t)+(b-\tau)(\nabla v,\nabla v_t)\Bigr)\,ds
+2\beta(t+t_0).
\end{equation}

We next differentiate once more. Testing the augmented system \eqref{eq:augmented} with $(u,w)$ yields
\begin{equation}\label{eq:psi_double}
\begin{aligned}
\Psi''(t)
&=
2\|u_t\|_2^2+2\|w_t\|_2^2+2(u,u_{tt})+2(w,w_{tt}) \\
&\quad
+2(\nabla u,\nabla u_t)+2(b-\tau)(\nabla v,\nabla v_t)+2\beta \\
&=
2\|u_t\|_2^2+2\|w_t\|_2^2
-2Q_\alpha(u,w)
+2\int_\Omega |u|^\gamma\ln|u|\,dx \\
&\quad
-2(b-\tau)(\nabla w,\nabla v_t)
+2(b-\tau)(\nabla v,\nabla v_t)
+2\beta.
\end{aligned}
\end{equation}

The last two terms yield the unfavorable residual $-2\tau(b-\tau)\|\nabla v_t\|_2^2$. To handle it, we reconstruct the logarithmic contribution through the exact energy identity. From the definition of $\mathcal{E}(t)$ in \eqref{eq:energy}, we obtain
\begin{equation}\label{eq:log}
2\int_\Omega |u|^\gamma\ln|u|\,dx
=
\gamma\|u_t\|_2^2
+\gamma\|w_t\|_2^2
+\gamma\tau(b-\tau)\|\nabla v_t\|_2^2
+\gamma Q_\alpha(u,w)
-2\gamma\mathcal{E}(t)
+\frac{2}{\gamma}\|u\|_\gamma^\gamma.
\end{equation}
Substituting \eqref{eq:log} into \eqref{eq:psi_double}, we arrive at
\begin{equation}
\begin{aligned}
\Psi''(t)
&=
(\gamma+2)\|u_t\|_2^2
+(\gamma+2)\|w_t\|_2^2
+(\gamma-2)\tau(b-\tau)\|\nabla v_t\|_2^2 \\
&\quad
+(\gamma-2)Q_\alpha(u,w)
-2\gamma\mathcal{E}(t)
+\frac{2}{\gamma}\|u\|_\gamma^\gamma
+2\beta.
\end{aligned}
\end{equation}
Thus the unfavorable MGT residual has been absorbed, and in fact converted into the positive term $(\gamma-2)\tau(b-\tau)\|\nabla v_t\|_2^2$, which is precisely where the assumption $\gamma>2$ enters.

Using the exact energy identity from Subsection~\ref{subsec:strong_continuity}, we further rewrite
\begin{equation*}
-2\gamma \mathcal{E}(t)
=
-2\gamma \mathcal{E}(0)
+
2\gamma\int_0^t
\Bigl(
\|\nabla u_t(s)\|_2^2
+
(b-\tau)\|\nabla v_t(s)\|_2^2
\Bigr)\,ds.
\end{equation*}
Consequently,
\begin{equation}
\begin{aligned}
\Psi''(t)
&=
(\gamma+2)\bigl(\|u_t\|_2^2+\|w_t\|_2^2\bigr)
+(\gamma-2)\tau(b-\tau)\|\nabla v_t\|_2^2 \\
&\quad
+2\gamma\int_0^t
\Bigl(
\|\nabla u_t(s)\|_2^2
+
(b-\tau)\|\nabla v_t(s)\|_2^2
\Bigr)\,ds \\
&\quad
+(\gamma-2)Q_\alpha(u,w)
+\frac{2}{\gamma}\|u\|_\gamma^\gamma
-2\gamma\mathcal{E}(0)
+2\beta.
\end{aligned}
\end{equation}

We now estimate $(\Psi'(t))^2$. Define
\begin{equation*}
X_1(t)
:=
\Psi(t)
-
(T-t)\Bigl(\|\nabla u_0\|_2^2+(b-\tau)\|\nabla v_0\|_2^2\Bigr),
\end{equation*}
and
\begin{equation*}
X_2(t)
:=
\|u_t(t)\|_2^2+\|w_t(t)\|_2^2
+\int_0^t
\Bigl(
\|\nabla u_t(s)\|_2^2
+
(b-\tau)\|\nabla v_t(s)\|_2^2
\Bigr)\,ds
+\beta.
\end{equation*}
Applying the Cauchy--Schwarz inequality to the terms in \eqref{eq:psi_single}, we find
\begin{equation*}
(\Psi'(t))^2
\le 4X_1(t)X_2(t)
\le 4\Psi(t)X_2(t).
\end{equation*}

Set
\begin{equation*}
\eta_0:=\frac{\gamma-2}{4}>0.
\end{equation*}
Then $4(1+\eta_0)=\gamma+2\le 2\gamma$, and therefore
\begin{equation}
\Psi''(t)-4(1+\eta_0)X_2(t)
\ge
(\gamma-2)Q_\alpha(u,w)
+\frac{2}{\gamma}\|u\|_\gamma^\gamma
-2\gamma\mathcal{E}(0)
-\gamma\beta.
\end{equation}

Since $\mathcal{I}_\alpha(u(t),w(t))<0$ by Proposition~\ref{prop:unstable_invariance}, the scaling argument used there provides $\lambda^*\in(0,1)$ such that
\begin{equation*}
\mathcal{I}_\alpha(\lambda^*u,\lambda^*w)=0.
\end{equation*}
Hence $(\lambda^*u,\lambda^*w)\in\mathcal{N}_\alpha$, so by the definition of $d_\alpha$ and the algebraic identity for $\mathcal{J}_\alpha$,
\begin{equation*}
(\gamma-2)Q_\alpha(u,w)+\frac{2}{\gamma}\|u\|_\gamma^\gamma
>
2\gamma d_\alpha.
\end{equation*}
It follows that
\begin{equation}
\Psi''(t)-4(1+\eta_0)X_2(t)
>
2\gamma(d_\alpha-\mathcal{E}(0))
-\gamma\beta.
\end{equation}
We now choose
\begin{equation*}
\beta:=2(d_\alpha-\mathcal{E}(0))>0.
\end{equation*}
With this choice,
\begin{equation*}
\Psi''(t)-4(1+\eta_0)X_2(t)\ge 0,
\end{equation*}
and therefore
\begin{equation*}
\Psi(t)\Psi''(t)-(1+\eta_0)(\Psi'(t))^2\ge 0
\qquad\text{for all } t\in[0,T].
\end{equation*}

Define
\begin{equation*}
S(t):=\Psi(t)^{-\eta_0}.
\end{equation*}
A direct computation shows that the above inequality is equivalent to $S''(t)\le 0$, so $S$ is concave on $[0,T]$. Hence,
\begin{equation*}
S(T)\le S(0)+S'(0)T.
\end{equation*}
To force a contradiction, we now choose the parameters so that the right-hand side becomes negative.

Set
\begin{equation*}
A:=\|\nabla u_0\|_2^2+(b-\tau)\|\nabla v_0\|_2^2,
\qquad
B:=\|u_0\|_2^2+\|w_0\|_2^2+\beta t_0^2.
\end{equation*}
Then
\begin{equation*}
\Psi(0)-\eta_0\Psi'(0)T
=
B+T\bigl(A-\eta_0\Psi'(0)\bigr).
\end{equation*}
Moreover,
\begin{equation*}
\Psi'(0)=2(u_0,u_1)+2(w_0,w_1)+2\beta t_0.
\end{equation*}
Since $\beta>0$ has already been fixed independently of $t_0$, we may choose $t_0>0$ sufficiently large so that
\begin{equation*}
A-\eta_0\Psi'(0)<0.
\end{equation*}
With this choice of $t_0$, the quantity $B$ is fixed, and we may then choose $T>0$ so large that
\begin{equation*}
T>\frac{B}{\eta_0\Psi'(0)-A}.
\end{equation*}
This yields
\begin{equation*}
\Psi(0)-\eta_0\Psi'(0)T<0,
\end{equation*}
and therefore
\begin{equation*}
S(T)<0.
\end{equation*}

This is impossible. Indeed, if the solution were global, then $Y\in C([0,T];\mathcal H)$ would imply $\Psi(T)<\infty$, while by definition $\Psi(t)>0$ for all $t\in[0,T]$. Hence
\begin{equation*}
S(T)=\Psi(T)^{-\eta_0}>0,
\end{equation*}
contradicting the previous conclusion. Therefore the assumption $T_{\max}=\infty$ is false, and thus
\begin{equation*}
T_{\max}<\infty.
\end{equation*}

It remains to identify the mechanism of blow-up. Suppose, for contradiction, that the gradient terms remain bounded up to $T_{\max}$, namely,
\begin{equation*}
\limsup_{t\uparrow T_{\max}}
\bigl(
\|\nabla u(t)\|_2^2+\|\nabla w(t)\|_2^2
\bigr)
\le M<\infty.
\end{equation*}
Then $u(t)$ is uniformly bounded in $H_0^1(\Omega)$ on $[0,T_{\max})$. Since $\gamma<2^*$, we may choose
\begin{equation*}
0 < \eta <\min\{\gamma-2,\,2^*-\gamma\}.
\end{equation*}
For the logarithmic primitive, we have
\begin{equation*}
|F(s)|\le C_\eta\bigl(|s|^{\gamma-\eta}+|s|^{\gamma+\eta}\bigr),
\qquad s\in\mathbb R.
\end{equation*}
By the Sobolev embeddings
\begin{equation*}
H_0^1(\Omega)\hookrightarrow L^{\gamma-\eta}(\Omega),
\qquad
H_0^1(\Omega)\hookrightarrow L^{\gamma+\eta}(\Omega),
\end{equation*}
the quantity $\int_\Omega |F(u(t))|\,dx$ remains uniformly bounded on $[0,T_{\max})$.

Returning to the exact energy identity and using $\mathcal{E}(t)\le \mathcal{E}(0)$, we obtain
\begin{equation*}
\frac12\|u_t\|_2^2
+\frac12\|w_t\|_2^2
+\frac{\tau(b-\tau)}{2}\|\nabla v_t\|_2^2
\le
\mathcal{E}(0)+\int_\Omega |F(u(t))|\,dx
\le C_M.
\end{equation*}
Thus all kinetic terms remain uniformly bounded as well. Combined with the assumed bound on $\|\nabla u(t)\|_2$ and $\|\nabla w(t)\|_2$, this shows that the full phase-space norm $\|Y(t)\|_{\mathcal H}$ stays uniformly bounded on $[0,T_{\max})$.

This contradicts the continuation principle in Proposition~\ref{prop:local}, which asserts that if $T_{\max}<\infty$, then
\begin{equation*}
\limsup_{t\uparrow T_{\max}}\|Y(t)\|_{\mathcal H}=\infty.
\end{equation*}
Therefore, the gradient terms cannot remain bounded, and we conclude that
\begin{equation*}
\limsup_{t\uparrow T_{\max}}
\bigl(
\|\nabla u(t)\|_2^2+\|\nabla w(t)\|_2^2
\bigr)
=
\infty.
\end{equation*}
This completes the proof.
\end{proof}

\end{document}